\documentclass{amsart}
\usepackage{amsmath}
\usepackage{amssymb}
\usepackage{color}
\usepackage{mathdots}

\usepackage{graphicx}
\usepackage{epsfig}
\usepackage{multirow}
\usepackage{colortbl}
\usepackage[table]{xcolor}

\definecolor{lightgray}{gray}{0.9}

\newtheorem{theorem}{Theorem}[section]
\newtheorem{lemma}[theorem]{Lemma}
\newtheorem{proposition}[theorem]{Proposition}
\newtheorem{corollary}[theorem]{Corollary}
\theoremstyle{definition}

\theoremstyle{remark}
\newtheorem{remark}[theorem]{Remark}
\numberwithin{equation}{section}

\def\separa{\hbox to 14 truecm{\hrulefill}}

\author[P. Danchev]{Peter Danchev}
\address{Institute of Mathematics and Informatics, Bulgarian Academy of Sciences, 1113 Sofia, Bulgaria}
\email{danchev@math.bas.bg}
\thanks{The first author was partially supported by the BIDEB 2221 of T\"UB\'ITAK}

\author[E. Garc\'\i a]{Esther Garc\'\i a}
\address{Departamento de Matem\'{a}tica  Aplicada, Ciencia e Ingenier\'{\i}a de los Materiales y Tecnolog\'{\i}a
Electr\'onica,
Universidad Rey Juan Carlos, 28933 M\'{o}s\-to\-les (Madrid), Spain}
\thanks{The second and third authors were partially supported by Ayuda Puente 2023, URJC, and MTM2017-84194-P (AEI/FEDER, UE)}
\email{esther.garcia@urjc.es}

\author[M. G\'omez Lozano]{Miguel G\'omez Lozano}
\address{Departamento de \'Algebra, Geometr\'{\i}a y
Topolog\'{\i}a, Universidad de M\'alaga, 29071 M\'alaga, Spain}
\thanks{The three authors were partially supported by the Junta de Andaluc\'{\i}a FQM264.}
\email{miggl@uma.es}

\begin{document}

\title[Decompositions of Periodic Matrices into a Sum of Special Matrices]{Decompositions of Periodic Matrices into a Sum of Special Matrices}
\maketitle

\begin{abstract} We study the problem of when a periodic square matrix of order $n$ over an arbitrary field $\mathbb{F}$ is decomposable into the sum of a square-zero matrix and a torsion matrix, and show that this decomposition can always be obtained for matrices of rank at least $\frac n2$ when $\mathbb{F}$ is either a field of prime characteristic, or the field of rational numbers, or an algebraically closed field of zero characteristic. We also provide a counterexample to such a decomposition when $\mathbb{F}$ equals the field of the real numbers.
Moreover, we prove that each periodic square matrix over any field is a sum of an idempotent
matrix and a torsion matrix.

\end{abstract}


\bigskip



\medskip

\section{Introduction}

The decomposition of matrices over an arbitrary field into the sum of some special elements, like nilpotent elements,
idempotent elements, potent elements, units, etc., was in the focus of many researchers for a long time (see, e.g.,
\cite{AT1}, \cite{AT2}, \cite{B}, \cite{BCDM}, \cite{BM}, \cite{DM}, \cite{Sh}, \cite{St1} and \cite{St2} and the bibliography cited therewith).

A square matrix $A$ is called {\it periodic} if $A^m=A^n$ for certain $m>n\ge 1$; when $n=1$, the matrix $A$ is called {\it $m$-potent} or just, for short, {\it potent}. Likewise, when $m=2$ and $n=1$, $A$ is known as an {\it idempotent} matrix. A major type of potent matrices are the so-called {\it torsion} matrices: invertible matrices $A$ with the property that $A^s={\rm Id}$ for some integer $s\geq 1$, where ${\rm Id}$ stands for the identity matrix. We also say that $A$ is a {\it nilpotent} matrix if there exists some integer $k>1$ such that $A^k=0$; if $k$ is the smallest number with this property, we call it {\it index of nilpotence}. In the case where $A^2=0$, the matrix $A$ is called {\it square-zero matrix}. It is obvious that both nilpotent and potent matrices are always periodic.

In \cite{DGL4}, we investigated the situation in which a square matrix of arbitrary size $n\geq 2$ over an arbitrary field could  be expressed as the sum of an invertible matrix and a  nilpotent matrix $N$ with $N^k$=0 for some fixed $k\geq 1$. It was shown there that this is always possible when the rank of the matrix is no less than $\frac{n}{k}$.

On the same vein, in \cite{DGL3} we showed that a nilpotent matrix of size $n$ over a division ring can be expressed as the sum of a torsion matrix and a square-zero matrix if, and only if, its rank is at least $\frac{n}{2}$. The key idea in \cite{DGL3} was to show that the torsion matrices satisfied certain characteristic polynomials. Extending that idea, in \cite{DGL5}, we studied when block matrices could be modified by adding a square-zero matrix to obtain prescribed characteristic polynomials. 
We shall freely use this fact in the sequel in order to decompose a periodic matrix into some special elements.

Throughout the text of the current article, we denote by $\mathbb{M}_n(R)$ the matrix ring consisting of all square matrices of size $n\times n$, where $n\geq 1$ is a natural number, over a ring $R$. As usual, the letters $\mathbb{F}$ and $\Delta$ are reserved to designate a field and a division ring, respectively.

\medskip

\section{Decomposing periodic matrices}

We start our work with the following preliminary comments.

\begin{remark}\label{periodicdescription}
A periodic matrix $A$ with $A^m=A^n$ for certain integers $m>n\ge 1$ clearly satisfies a polynomial of the form $x^{n}(x^{m-n}-1)$. Therefore, the elementary divisors of such matrix are of the form $x$, $x^k$ (with $k\ge 2$) and/or divisors of the polynomial $x^{m-n}-1$; thus, $A$ is similar to a matrix of the form
$$
\left(
  \begin{array}{c|c|c|c|c}
    {0}& 0&0&0&0\\
    \hline
    0& N_{k_1}&0&0 & 0 \\
    \hline
 0& 0&\ddots&0 & 0\\
    \hline
     0& 0&0&N_{k_s} & 0\\
    \hline
  0&0& 0& 0 & T \\
  \end{array}
\right),
$$
where ${0}$ in the left-upper block corresponds to the elementary divisors of the form $x$, the elements
$N_{k_1},\dots,N_{k_s}$ correspond to the elementary divisors of the form $x^{k_i}$ (with $k_i\ge 2$), and $T$ contains in its main diagonal the (invertible) companion matrices associated to factors of $x^{m-n}-1$ (so, one inspects that $T$ is a torsion matrix).
\end{remark}

Thanks to this decomposition into special blocks, surprisingly with no too many efforts, we can express each periodic matrix as the sum of an idempotent matrix and a torsion matrix. Concretely, the following is true:

\begin{proposition}\label{tmaproblem2} Every periodic matrix over a field is the sum of an idempotent matrix and a torsion matrix.
\end{proposition}

\begin{proof}
As explained in Remark \ref{periodicdescription} quoted above, any periodic matrix $A$ with $A^m=A^n$ for certain integers $m>n\ge 1$ is similar to a matrix of the form
$$
\left(
  \begin{array}{c|c|c|c|c}
    0& 0&0&0&0\\
    \hline
    0& N_1&0&0 & 0 \\
    \hline
 0& 0&\ddots&0 & 0\\
    \hline
     0& 0&0&N_s & 0\\
    \hline
  0&0& 0& 0 & T \\
  \end{array}
\right).
$$
Now, we can do the following two things:
\begin{itemize}

\item To express the left-upper block as $(0)=({\rm Id})+( -{\rm Id})$ (i.e., idempotent + torsion of order $2$).

\item To decompose each nilpotent component $N_{k_i}$ as

$$
N_{k_i}=\underbrace{\left(
          \begin{array}{cccc}
            0 & 0 & 0 & 1 \\
            0 & 0 & 0 & 1 \\
            \vdots  &    &   & \vdots \\
            0 & 0 & 0 & 1 \\
          \end{array}
        \right)}_{\rm idempotent}+\underbrace{\left(
                  \begin{array}{cccc}
                   0 & 0 & 0 & -1 \\
                    1 & 0 & 0 & -1 \\
                    0 & \ddots &  & \vdots \\
                    0 & 0 & 1 & -1 \\
                  \end{array}
                \right)}_{\hbox{ torsion of index } k_i+1},
$$

\medskip

\noindent where the last matrix is torsion, because it is the companion matrix of the polynomial
$x^{k_i}+x^{k_i-1}+\dots+1$, which is a factor of the polynomial $x^{k_i+1}-1$.

\item Write $T=(0)+T$ (i.e., idempotent + torsion).
\end{itemize}

\medskip

This concludes the proof after all.
\end{proof}

As a direct consequence, we yield the following.

\begin{corollary}\label{nilpotent2} Every nilpotent matrix over a division ring is the sum of an idempotent matrix and a torsion matrix.
\end{corollary}

\begin{proof} Knowing the classical fact that each nilpotent matrix over a division ring $\Delta$ is similar to its Jordan canonical form $J$, which is indeed a matrix over the field ${\rm Z}(\Delta)$, that is, the center of $\Delta$, we then may apply Proposition~\ref{tmaproblem2} to obtain the wanted decomposition.
\end{proof}

In this aspect, an important query which immediately arises is of whether or {\it not} Proposition~\ref{tmaproblem2} could be expanded to division rings, i.e., whether or {\it not} Corollary~\ref{nilpotent2} remains valid for periodic matrices. However, even for potent matrices over division rings the situation seems to be rather complicated.

\medskip

Inspired by this decomposition and trying to generalize our work \cite{DGL3}, where we decomposed nilpotent matrices of order $n$ and rank at least $\frac n2$ as the sum of a torsion matrix and a square-zero matrix, we pose the following question.

\medskip

{\bf Question:} {\it Is any periodic matrix representable as the sum of a torsion matrix and a square-zero matrix?}

\medskip

The rest of this section will be devoted to study this problem. To this target, we will make use of the description of periodic matrices in Remark \ref{periodicdescription} alluded to above and of our key result from \cite{DGL3}.

\begin{lemma}\label{nilpotent}\cite[Proposition 2.1]{DGL3}
If a nilpotent block $N_{k}$ of size $k\ge 2$ is followed by $s$ blocks corresponding to elementary divisors of the form $x$ such that $0\le s\le k-2$, then the resulting matrix can be decomposed into the sum of a torsion matrix and a square-zero matrix.
\end{lemma}

The following statement deals with the special situation when $k$ blocks corresponding to elementary divisors of the form $x$ are combined with a torsion block of order $k$. It is rather straightforward, and so we omit its proof.

\begin{lemma}\label{maximal}
If $A\in\mathbb{M}_{2k}(\mathbb{F})$ is a block matrix of the form
$$
A=\left(
    \begin{array}{c|c}
      0 & 0 \\
      \hline
      0 & T \\
    \end{array}
  \right)
$$
for some $T\in\mathbb{M}_{k}(\mathbb{F})$, then
$$
N=\left(
    \begin{array}{c|c}
      -T & -T \\
      \hline
      T & T \\
    \end{array}
  \right)
$$
is a square-zero matrix and
$$
(A-N)^6=\left(
    \begin{array}{c|c}
      T^6 & 0 \\
      \hline
      0 & T^6 \\
    \end{array}
  \right).
$$
In particular, if $T$ is torsion, the matrix $A$ decomposes as the sum of a torsion matrix and a square-zero matrix.
\end{lemma}

The following result, which was proven in \cite{DGL5}, will also be very useful when combining blocks of $0$'s  with torsion companion matrices:

\begin{lemma}\label{teoremappal}\cite[Theorem 2.3]{DGL5}
Let $\mathbb{F}$ be a field, let $n,k\in \mathbb{N}$ with $k<n-k$, and consider the block matrix
$$
A=\left(
         \begin{array}{c|c}
           {\bf 0}_{k,k} & {\bf 0}_{k,n-k} \\
           \hline
           {\bf 0}_{n-k,k} & A_{22} \\
         \end{array}
       \right)\in \mathbb{M}_{n}(\mathbb{F})
$$
consisting of $k$ rows and columns of zeros, and of an invertible non-derogative matrix $A_{22}$. Then, for any monic polynomial $q(x)$ of degree $n$ whose trace coincides with the trace of $A$, there exists a square-zero matrix $N$ such that the characteristic polynomial of $A+N$ coincides exactly with $q(x)$.
\end{lemma}

Suppose now that $\mathbb{F}=\mathbb{Q}$, the field of all rational numbers, and let us deal with a block matrix consisting of some blocks corresponding to elementary divisors of the form $x$ and the companion matrix associated to a cyclotomic polynomial of degree $\ge 2$. Recall that the trace of every cyclotomic polynomial in $\mathbb{Q}[x]$ is always $0$, $1$ or $-1$ (just by considering the  M\"{o}bius function). Thus, the decomposition of such matrices into torsion and square-zero is a direct consequence of Lemma \ref{teoremappal} (in fact, we only need to choose the appropriate characteristic polynomial according to the trace of the cyclotomic polynomial). Specifically, the following technicality is valid:

\begin{lemma}\label{cyclotomic}
If $p(x)\in \mathbb{Q}[x]$ is a cyclotomic polynomial of degree $n-k\ge 2$, $k<n-k$, and
$$
A=\left(
  \begin{array}{c|c}
    0 & 0\dots 0 \\
    \hline
     0 & \\
    \vdots & C(p(x))\\
     0 &
  \end{array}
\right)\in \mathbb{M}_{n}(\mathbb{F}),
$$
then:
\begin{itemize}
\item[(i)] if the trace of $p(x)$ is $0$, there exists a square-zero matrix $N$ such that the characteristic polynomial of $A+N$ is $x^{n}-1$; in particular, $(A+N)^{n}={\rm Id}$;
\item[(ii)] if the trace of $p(x)$ is $1$, there exists a square-zero matrix $N$ such that the characteristic polynomial of $A+N$ is $x^{n}-x^{n-1}+x^{n-2}-\dots+(-1)^n$; in particular, if $n$ is even, we have $(A+N)^{2(n+1)}={\rm Id}$, and if $n$ is odd, we have $(A+N)^{n+1}={\rm Id}$;
\item[(iii)] if the trace of $p(x)$ is $-1$, there exists a square-zero matrix $N$ such that the characteristic polynomial of $A+N$ is $x^{n}+x^{n-1}+x^{n-2}+\dots+1$; in particular, $(A+N)^{n+1}={\rm Id}$.
\end{itemize}
\end{lemma}

Now, we are able to prove the main result of this paper which generalizes our recent result in \cite{DGL3} (compare with \cite{DGL4} as well) and gives a satisfactory necessary and sufficient condition for a square matrix over certain fields to be a sum of a torsion matrix and a nilpotent matrix of the smallest index of nilpotence $2$.

\begin{theorem}\label{tmappalperidic} Let $\mathbb{F}$ be either a field of prime characteristic, or $\mathbb{F}=\mathbb{Q}$, or an algebraically closed field of characteristic zero. Then, every periodic matrix $A$ over $\mathbb{F}$ is decomposable as the sum of a torsion matrix and a square-zero matrix if, and only if, the rank of $A$ is at least $\frac{n}{2}$.
\end{theorem}

\begin{proof}
Since all torsion matrices have full rank, and the rank of a square-zero matrix is at most $\frac n2$, the necessary condition is quite obvious.

Suppose now that $A^m=A^n$ for some integers $m>n\ge 1$ and that the rank of $A$ is at least $\frac n2$; following Remark \ref{periodicdescription} stated above, the matrix $A$ is similar to a matrix of the form
$$
\left(
  \begin{array}{c|c|c|c|c}
    0& 0&0&0&0\\
    \hline
    0& N_{k_1}&0&0 & 0 \\
    \hline
 0& 0&\ddots&0 & 0\\
    \hline
     0& 0&0&N_{k_s} & 0\\
    \hline
  0&0& 0& 0 & T \\
  \end{array}
\right);
$$
and it must be that: (1) when $\mathbb{F}=\mathbb{Q}$, the elementary divisors which are factors of $x^{m-n}-1$ are (different) cyclotomic polynomials; (2) when $\mathbb{F}$ is algebraically closed of characteristic zero, all of the factors of $x^{m-n}-1$ are (different) degree-one polynomials of the form $x-\lambda$, where $\lambda$ is a root of the unity.

Furthermore, since the rank of $A$ is greater than or equal to $\frac n2$, we can distribute the 0-blocks of order one in the following way:
\begin{itemize}
\item[(a)] follow each $N_{k_i}$ of size $k_i\ge 2$ by $s_i$-blocks corresponding to elementary divisors of the form $x$,     where $0\le s_i\le k_i-2$; then, decompose $N_{k_i}$ and the following $s_i$ Jordan blocks of the form $(0)$ into a     torsion matrix and a square-zero matrix applying Lemma~\ref{nilpotent}.

\medskip

\noindent Depending on which field we are dealing with, we differ the next three cases:

\medskip

\item[(b.1)] when $\mathbb{F}$ is an algebraically closed field, follow each companion matrix associated to $x-\lambda$ ($\lambda$ being a root of the unity) by $0\le s\le 1$ blocks corresponding to elementary divisors of the form $x$; when $s=1$ decompose into a torsion matrix and a square-zero matrix employing Lemma~\ref{maximal}, and when $s=0$ just take $N=(0)$;

\item[(b.2)] when $\mathbb{F}=\mathbb{Q}$, follow each companion matrix of size $m_i$ in the main diagonal of $T$
    (associated to a cyclotomic polynomial) by $t_i$-blocks corresponding to elementary divisors of the form $x$, where $0\le t_i\le m_i$; when $t_i=m_i$ decompose into a torsion matrix and a square-zero matrix utilizing
    Lemma~\ref{maximal}, and when $t_i<m_i$ use Lemma~\ref{cyclotomic} to decompose into a torsion matrix and a square-zero matrix;

\item[(b.3)] when the characteristic of $\mathbb{F}$ is a prime $p$, suppose that the whole block $T$ has size $t\ge 0$ and follow it by $r$-blocks corresponding to elementary divisors of the form $x$, where $0\le r\le t$, and let us call it $T_1$; since the minimal polynomial of $T_1$ is algebraic over the prime field $\mathbb{F}_p$, because it divides $x^r(x^{m-n}-1)$, the characteristic polynomial of $T_1$ is also algebraic over $\mathbb{F}_p$ and, therefore, we can exploit \cite[Theorem 1.8]{DGL3} to decompose $T_1$ as the sum of a torsion matrix and a square-zero matrix.
\end{itemize}
This finishes the proof after all.
\end{proof}

The next construction shows that the above result is {\it not} true in general for arbitrary fields.

\begin{remark}\label{couterexamplechar0}
When dealing with non-algebraically closed fields of characteristic zero other than $\mathbb{Q}$, e.g., the field $\mathbb{R}$ consisting of all real numbers, the decomposition of a periodic matrix into a torsion matrix and a square-zero matrix does {\it not} always hold. For instance, if we consider the matrix
$$
A=\left(
    \begin{array}{c|cc}
      0 & 0  &0 \\
      \hline
      0 & 0 &-1 \\
      0 & 1 &-\sqrt{2} \\
    \end{array}
  \right)\in \mathbb{M}_3(\mathbb{R}),
$$
we can easily check that $A$ is periodic, because $A^9=A$ holds; nevertheless, there does {\it not} exist a square-zero matrix $N\in \mathbb{M}_3(\mathbb{R})$ such that $A+N$ is torsion: indeed, otherwise, the characteristic polynomial of $A+N$ would be of the form $$p(x)=x^3+\sqrt{2}x^2+ax+b\in \mathbb{R}[x]$$ for some $a,b\in \mathbb{R}$ and, moreover, such polynomial should divide $x^n-1$ for some $n\ge 3$; its three roots $\alpha_1,\alpha_2,\alpha_3\in \mathbb{C}$ should be $n^{\rm th}$-roots of the unity (in fact, one of them, say $\alpha_1$, a real number, hence $\alpha_1=\pm 1$, and the other two, say $\alpha_2,\alpha_3$, complex conjugate roots of unity) with $\alpha_1+\alpha_2+\alpha_3=-\sqrt{2}$. Since the sum of the two roots of unity $\alpha_2$ and $\alpha_3$ cannot be smaller that $ -2$, one deduces that $\alpha_1=-1$ and $\alpha_2+\alpha_3=-\sqrt{2}+1$, and this automatically gives that $$p(x)=(x+1)(x^2+(\sqrt{2}-1)x+1).$$ Next, solving the
degree $2$ equation $x^2+(\sqrt{2}-1)x+1=0$ in $\mathbb{C}$ directly leads us to
$$\alpha_2,\alpha_3=\frac{1}{2}\left(1-\sqrt{2}\pm\sqrt{-1-\sqrt{2}}\right)\in \mathbb{C}.$$ However, the minimal polynomial of $\alpha_2$ (and also of $\alpha_3$) over $\mathbb{Q}$ is $x^4-2x^3+x^2-2x+1$, which is definitely {\it not} cyclotomic because its trace equals $2$, thus proving that $\alpha_2,\alpha_3$ cannot be roots of unity, a contradiction. This ends the example and our considerations.
\end{remark}

The logical reason why the matrix $A$ in the above counterexample cannot be decomposed into the sum of a torsion matrix and a square-zero matrix relies on the fact that there are {\it no} polynomials in $\mathbb{R}[x]$ whose trace is $\sqrt{2}-1$ and which divide the polynomial $x^r-1$ for some $r\in\mathbb{N}$.

Nevertheless, when $A$ is of the special form
$$
A=\left(
    \begin{array}{c|c}
      {\bf 0}_{k,k} & {\bf 0}_{k,n-k} \\
      \hline
      {\bf 0}_{n-k,k} & C(p(x)) \\
    \end{array}
  \right)\in \mathbb{M}_n(\mathbb{R})
$$
with $k\ge 2$ and $C(p(x))$ is a torsion companion matrix of an irreducible polynomial, then there will exist some
polynomial $a(x)\in\{x^k-1,x^k+1\}$ such that the resulting polynomial $p(x)a(x)$ has the same trace as $p(x)$ and divides $x^r-1$ for some $r\in \mathbb{N}$. Precisely, we offer to establish the following:

\begin{proposition}\label{cyclotomic2}
Let $\mathbb{F}$ be a field of characteristic zero, and let $A$ be a matrix of the form
$$
A=\left(
    \begin{array}{c|c}
      {\bf 0}_{k,k} & {\bf 0}_{k,n-k} \\
      \hline
      {\bf 0}_{n-k,k} & C(p(x)) \\
    \end{array}
  \right)\in \mathbb{M}_n(\mathbb{R})
$$
with $k\ge 2$ and $C(p(x))$ is a torsion companion matrix of an irreducible polynomial $p(x)$ such that $(C(p(x)))^r={\rm Id}$ for some $r\in \mathbb{N}$. Then,
\begin{itemize}
\item[(i)] if the greatest common divisor of $p(x)$ and $x^k-1$ is $1$, there exists a square-zero matrix $N$ such that the characteristic polynomial of $A+N$ is $p(x)(x^k-1)$; in this case, $A+N$ is a torsion matrix with $(A+N)^{m}={\rm Id}$, where $m$ is the least common multiple of $k$ and $r$.
\item[(ii)] if the greatest common divisor of $p(x)$ and $x^k+1$ is $1$, there exists a square-zero matrix $N$ such that the characteristic polynomial of $A+N$ is $p(x)(x^k+1)$; in this case, $A+N$ is a torsion matrix with $(A+N)^{m}={\rm Id}$, where $m$ is the least common multiple of $2k$ and $r$.
\end{itemize}
\end{proposition}

\begin{proof}
One observes that the polynomial $p(x)$ is either co-prime with $x^k-1$ or with $x^k+1$, because it is irreducible. But, since $k\ge 2$, the trace of the polynomials $p(x)(x^k-1)$ and $p(x)(x^k+1)$ coincides with the trace of $p(x)$, and we can just take into account Lemma \ref{nilpotent} to obtain the square-zero matrices $N$ such that $A+N$ have the desired characteristic polynomials.
\end{proof}

We close our work with the following challenging problem:

\medskip

\noindent{\bf Problem:} Find a criterion when a square matrix over an arbitrary field is a sum of a torsion matrix, an idempotent matrix and a square-zero matrix.

\medskip
\medskip
\medskip

\noindent{\bf Funding:} The first-named author (Peter Danchev) was supported in part by the BIDEB 2221 of T\"UB\'ITAK; the second and third-named authors (Esther Garc\'{\i}a and Miguel G\'omez Lozano) were partially supported by Ayuda Puente 2023, URJC, and MTM2017-84194-P (AEI/FEDER, UE). The all three authors were partially supported by the Junta de Andaluc\'{\i}a FQM 264.

\vskip3.0pc

\bibliographystyle{plain}

\end{document}